\documentclass[graybox]{svmult}

\usepackage{mathptmx}
\usepackage{helvet}
\usepackage{courier}

\usepackage{makeidx}
\usepackage{graphicx}

\usepackage{multicol}
\usepackage[bottom]{footmisc}

\usepackage{amsmath}
\usepackage{amsfonts}
\usepackage{amssymb}
\usepackage{algorithm2e}
\usepackage{enumerate}
\usepackage{tikz}

\usepackage{dcolumn}
\newcolumntype{d}[1]{D{.}{.}{#1}}

\usepackage{url}
\DeclareMathOperator*{\tr}{tr}
\DeclareMathOperator*{\rank}{rank}
\DeclareMathOperator*{\Fix}{Fix}
\DeclareMathOperator*{\diag}{diag}
\newcommand{\setto}{\rightrightarrows}
\newcommand{\angstrom}{{\mbox{\normalfont\AA}}}

\makeindex
\smartqed

\graphicspath{{reconPicsColor/}} 

\begin{document}

\author{Jonathan M. Borwein \and Matthew K. Tam}
\title{Reflection methods for inverse problems with applications to protein conformation determination}
\titlerunning{Reflection methods for inverse problems}
\authorrunning{J.M. Borwein \and M.K. Tam}
\institute{J.M. Borwein \at CARMA Centre, University of Newcastle, Callaghan, NSW 2308, Australia. \email{jon.borwein@gmail.com}
          \and
          M.K. Tam \at CARMA Centre, University of Newcastle, Callaghan, NSW 2308, Australia. \email{matthew.tam@uon.edu.au}}

\maketitle

\abstract{The Douglas--Rachford reflection method is a general purpose algorithm useful for solving the feasibility problem of finding a point in the intersection of finitely many sets. In this chapter we demonstrate that applied to a specific problem, the method can benefit from heuristics specific to said problem which exploit its special structure. In particular, we focus on the problem of protein conformation determination formulated within the framework of matrix completion, as was considered in a recent paper of the present authors.
}

\section{Techniques of Variational Analysis}

This chapter builds on a series of seven lectures titled \emph{Techniques of Variational Analysis} given by the first author at the CIMPA school \emph{Generalized Nash Equilibrium Problems, Bilevel Programming and MPEC} held November 25 to December 6, 2013, University of Delhi, New Delhi, India. In this written presentation we focus on \emph{reflection methods} for \emph{protein conformation determination}, as was discussed in the seventh and final lecture of the series. The complete lectures --- one through six taken from \cite{tova} --- can be found online at:
 \vspace{-1.5ex}\begin{center}
  \url{http://www.carma.newcastle.edu.au/jon/ToVA/links.html}
 \end{center}

Before turning our attention to reflection methods, we briefly outline the content of the first six lectures.
 \begin{itemize}
 \item \textbf{Lectures 1 \& 2} provided an introduction to variational analysis and variational principles \cite[\S1-\S2]{tova}.
 \item \textbf{Lectures 3 \& 4} introduced nonsmooth analysis: normal cones and subdifferentials of lower semi-continuous functions, Fr\'echet and limiting calculus \cite[\S3.1-\S3.4]{tova}, and discussed convex functions and their calculus
rules \cite[\S4.1-\S4.4]{tova}.
 \item \textbf{Lecture 5} turned to multifunction analysis: sequences of sets, continuity of maps, minimality and maximal monotonicity, and distance functions \cite[\S5.1-\S5.3]{tova}.
 \item \textbf{Lecture 6} focussed on convex feasibility problems and the method of alternating projections \cite[\S4.7]{tova}, and therefore providing the preliminary background for the rest of this chapter.
 \end{itemize}

\section{Introduction to Reflection Methods}

Given a (finite) family of sets, the corresponding \emph{feasibility problem} is to find a point contained in their intersection. \emph{Douglas--Rachford reflection methods} form a class of general-purpose iterative algorithms which are useful for solving such problems. At each iteration, these methods perform \emph{(metric) reflections} and \emph{(metric/nearest point) projections} with respect to the individual constraint sets in a prescribed fashion. Such methods are most useful when applied to feasibility problems whose constraint sets have more easily computable reflections and projections than does the intersection.

When the underlying constraint sets are all convex, Douglas--Rachford methods are relatively well understood \cite{BCL04finding,BT14cyclic,BT13cyclic2,BNP14linear} --- their behaviour can be analysed using nonexpansivity properties of convex projections and reflections. In the absence of convexity, recent result have assumed the constraint sets to possess other structural and regularity properties \cite{BS11dr,AB13dr,HL13nonconvex}. However, at present, this developing theoretical foundation is not sufficiently rich to explain many of the successful applications in which one or more of the constraint sets lacks convexity \cite{ABT14matrix,ABT04comb,ERT07searching,GE08divide}. In these cases, the method can be viewed as a heuristic inspired by its behaviour within fully convex settings.

More generally, with any algorithm there is typically a trade-off between the scope of their applicability and tailoring of performance to particular instances. Douglas--Rachford reflection methods are no different. Owing to these methods' broad applicability, potential for further problem specific refinements when applied to special classes of feasibility problems is possible.

In this chapter, we investigate and develop one such refinement with a focus on application of the Douglas--Rachford method to \emph{protein conformation determination}. This application was previously considered as part of \cite{ABT14matrix}. We now propose problem specific heuristics, and also study the effect of increasing problem size. We finish by demonstrating a complementary application of the approach arising in the context of \emph{ionic liquid chemistry}.

The remainder of this chapter is organized as follows. In Sections~\ref{s:prelim}, \ref{s:edm}, \ref{s:DR} \& \ref{s:protein} we introduce the necessary mathematical preliminaries along with the Douglas--Rachford reflection method, before formulating the protein conformation determination problem. Substantial numerical and graphical results are given in Section~\ref{s:results}, and concluding remarks in Section~\ref{s:conclusion}.

\section{Mathematical Preliminaries}\label{s:prelim}
Let $\mathbb{E}$ denote a Euclidean space, that is, a finite dimensional Hilbert space. We will mainly be concerned with the space $\mathbb{R}^{m\times m}$ (i.e., real $m\times m$ matrices) equipped with the inner-product given by
 $$\langle A,B\rangle:=\tr(A^TB).$$
Here the symbol $\tr(X)$ (resp. $X^T$) denotes the trace (resp. transpose) of the matrix $X$. The induced norm is the \emph{Frobenius norm} and can be expressed as
 $$\|A\|_F:=\sqrt{\tr(A^TA)}=\sqrt{\sum_{i=1}^m\sum_{j=1}^ma_{ij}^2}.$$
The subspace of real symmetric $m\times m$ matrices is denoted $S^m$, and the cone of positive semi-definite $m\times m$ matrices by $S^m_+$.

Given sets $C_1,C_2,\dots,C_N\subseteq\mathbb{E}$, the \emph{feasibility problem} is
\begin{equation}\label{eq:feasibility}
 \text{find~}x\in\bigcap_{i=1}^NC_i.
\end{equation}
When the intersection in \eqref{eq:feasibility} is empty, one often seeks a ``good" surrogate for a point in the intersection. When $N=2$, a useful surrogate is a pair of points, one from each set, which minimize the distance between the sets -- a \emph{best approximation pair} \cite{BCL04finding}.

\section{Matrix Completion}\label{s:edm}

A \emph{partial (real) matrix} is an $m\times m$ array for which entries only in certain locations are known. Given a partial matrix $A=(a_{ij})\in\mathbb{R}^{m\times m}$, a matrix $B=(b_{ij})\in\mathbb{R}^{m\times m}$ is a \emph{completion} of $A$ if $b_{ij}=a_{ij}$ whenever $a_{ij}$ is known. The problem of \emph{(real) matrix completion} is the following: \emph{Given a partial matrix find a completion belonging to a specified family of matrices}.

Matrix completion can be naturally formulated as a feasibility problem. Let $A$ be the partial matrix to be completed. Choose $C_1,C_2,\dots,C_N$ such that their intersection is equal to the intersection of completions of $A$  with the specified matrix family. Then (\ref{eq:feasibility}) is precisely the problem of matrix completion for $A$. The simplest such case is when $C_1$ is the set of all completions of $A$ and the intersection of $C_2,\dots,C_N$ equals the desired matrix class.

\begin{remark}
 More generally, one may profitably consider matrix completion for rectangular matrices \cite{ABT14matrix}, for example with doubly stochastic matrices. However, since the partial matrices in the discussed protein application are always square, for the purposes of this discussion, we only concern ourselves with the square case.
\end{remark}

\section{The Douglas--Rachford Reflection Method}\label{s:DR}
The \emph{projection onto $C\subseteq\mathbb{E}$} is the set-valued mapping $P_C:\mathbb{E}\setto C$ which maps any point $x\in \mathbb{E}$ to its sets of nearest points in $C$. More precisely,
 \begin{equation*}
  P_C(x)=\left\{c\in C:\|x-c\|\leq\inf_{y\in C}\|x-y\|\right\}.
 \end{equation*}
The \emph{reflection with respect to $C$} is the set-valued mapping $R_C:\mathbb{E}\setto\mathbb{E}$ given by $ R_C=2P_C-I$, where $I$ denotes the identity mapping.

When $C$ is non-empty, closed, and convex, its corresponding projection operator (and hence its reflection) is single-valued (see, for example, \cite[Ch.~1.2]{C12iterative}).

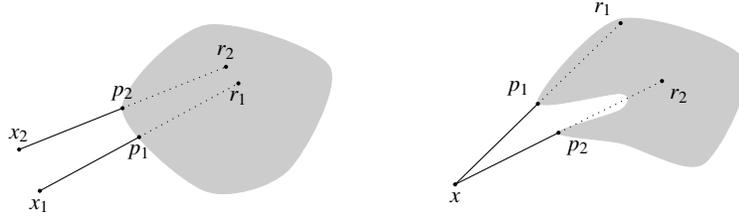
\begin{figure}
 \begin{center}
 \begin{minipage}{0.48\textwidth}
   \begin{center}
    \begin{tikzpicture}[scale=0.55]
      \filldraw [black!20] plot [smooth cycle] coordinates
         {(0,0) (2,0.5) (3,3) (0,4) (-2,2)};
      \fill (-4,0) circle (0.05) node[below] {$x_1$};
      \draw (-1.6,1.3) -- (-4,0);
      \fill (-1.6,1.3) circle (0.05) node[below]  {$p_1$};
      \fill[black] (0.8,2.6) circle (0.05) node[below] {$r_1$};
      \draw[dotted,black] (-1.6,1.3) -- (0.8,2.6);
      \fill (-4.5,1) circle (0.05) node[above] {$x_2$};
      \draw (-2.0,2) -- (-4.5,1);
      \fill (-2.0,2) circle (0.05) node[above]  {$p_2$};
      \fill[black] (0.5,3) circle (0.05) node[above] {$r_2$};
      \draw[dotted,black] (-2.0,2) -- (0.5,3);
      \draw (0,4.8) node {};
    \end{tikzpicture}
   \end{center}
 \end{minipage}
 \begin{minipage}{0.48\textwidth}
   \begin{center}
   \begin{tikzpicture}[scale=0.55]
     \filldraw [black!20] plot [smooth cycle] coordinates
        {(0,1) (2,0.5) (3,3) (0,4) (-2,2) (0,2.2)  (0,1.8) (-1.5,1.25)};
     \fill (-4,0) circle (0.05) node[below] {$x$};
     \fill (-2,1.95) circle (0.05) node[above left]  {$p_1$};
     \fill (-1.5,1.25) circle (0.05) node[below right]  {$p_2$};
     \draw (-4,0) -- (-2,1.95);
     \draw (-4,0) -- (-1.5,1.25);
     \draw[dotted,black] (-2,1.95) -- (0,3.9);
     \draw[dotted,black] (-1.5,1.25) -- (1,2.5);
     \fill[black] (0,3.9) circle (0.05) node[above left] {$r_1$};
     \fill[black] (1,2.5) circle (0.05) node[below right] {$r_2$};
   \end{tikzpicture}
   \end{center}
 \end{minipage}
 \caption{(Left) The (single-valued) projection, $p_i$, and reflection, $r_i$, of the point $x_i$ onto a convex set, for $i=1,2$. (Right) The (set-valued) projection, $\{p_1,p_2\}$, and reflection, $\{r_1,r_2\}$, of the point $x$ onto a non-convex set. Note the non-expansivity of the reflection in the convex case. }
 \end{center}
\end{figure}

Given $A,B\subseteq\mathbb{E}$ and $x_0\in\mathbb{E}$, the \emph{Douglas--Rachford reflection method} is the fixed point iteration given by
 \begin{equation}\label{eq:iterationDR}
  x_{n+1}\in T_{A,B}x_n\text{~where~}T_{A,B}=\frac{I+R_{B}R_{A}}{2}.
 \end{equation}
We refer to the sequence $(x_n)_{n=1}^\infty$ as a \emph{Douglas--Rachford sequence}, and to the mapping $T_{A,B}$ as the \emph{Douglas--Rachford operator}.

We now recall the behavior of the Douglas--Rachford method in the classical convex setting. In this case, $T_{A,B}$ is single-valued as a consequence of the single-valuedness of each of $P_A,P_B,R_A$ and $R_B$. We denote the set of \emph{fixed points} of a single-valued mapping $T$ by $\Fix T=\{x\in\mathbb{E}:Tx=x\}$, and the \emph{normal cone} of a convex set $C$ at the point $x$ by
  \begin{equation*}
   N_C(x) = \begin{cases}
             \{y\in\mathbb{E}:\langle C-x,u\rangle\leq 0\} & \text{if }x\in C,\\
             \emptyset & \text{otherwise.} \\
            \end{cases}
  \end{equation*}
 For convenience, we also introduce the two sets
  \begin{align*}
   E&=\left\{x\in A:\inf_{a\in A}\|a-x\|\leq \inf_{a\in A,b\in B}\|a-b\|\right\},\\
   F&=\left\{x\in B:\inf_{b\in B}\|x-b\|\leq \inf_{a\in A,b\in B}\|a-b\|\right\},
  \end{align*}
 and the vector $v=P_{\overline{B-A}}(0)$. Here the overline denotes the closure of the set.
\begin{theorem}[Convex Douglas--Rachford in finite dimensions {\cite{BCL04finding}}]\label{th:convexDR}
 Suppose ${A,B\subseteq\mathbb{E}}$ are closed and convex. For any $x_0\in\mathbb{E}$ define $x_{n+1}=T_{A,B}x_n$.  Then there is some $v \in\mathbb{E}$ such that:
 \begin{enumerate}[(i)]
  \item $x_{n+1}-x_n=P_BR_Ax_n-P_Ax_n\to v$ and $P_BP_Ax_n-P_Ax_n\to v$.
  \item If $A\cap B\neq\emptyset$ then $(x_n)_{n=1}^\infty$ converges to a point in $$\Fix(T_{A,B})=(A\cap B)+N_{\overline{A-B}}(0);$$ otherwise, $\|x_n\|\to+\infty$.
  \item Exactly one of the following two alternatives holds.
   \begin{enumerate}[(a)]
    \item $E=\emptyset$, $\|P_Ax_n\|\to+\infty$, and $\|P_BP_Ax_n\|\to+\infty$.
    \item $E\neq\emptyset$, the sequences $(P_Ax_n)_{n=1}^\infty$ and $(P_BP_Ax_n)_{n=1}^\infty$ are bounded, and their cluster points belong to $E$ and $F$, respectively; in fact, the cluster points of
     \begin{equation*}
      ((P_Ax_n,P_BR_Ax_n))_{n=1}^\infty \text{ and }  ((P_Ax_n,P_BP_Ax_n))_{n=1}^\infty
     \end{equation*}
    are a best approximation pairs relative to $(A,B)$.
   \end{enumerate}
 \end{enumerate}
\end{theorem}

\begin{figure}
 \begin{center}
 \begin{tikzpicture}[scale=0.9]
   \fill[black!20] (0,0) circle (1.75cm);
   \draw[<->,black!75,thick] (-4,0.66) -- (4,0.66);
   \fill (2.625,-1) node (x) {} circle (0.03) node[below left] {$x_n$};
   \fill (0.646,-0.246) node (Rx) {} circle (0.03) node[left] {$R_{A}x_n$};
   \draw[->,dotted] (x) -- (Rx);
   \fill (0.646,1.566) node (RRx) {} circle (0.03) node[left] {$R_{B}R_{A}x_n$};
   \draw[->,dotted] (Rx) -- (RRx);
   \fill (1.6355,0.283) node (Tx) {} circle (0.03) node[right] {$x_{n+1}=T_{A,B}x_n$};
   \draw[->] (x) -- (Tx);
   \draw (-1,-1) node {$A$};
   \draw (-2.5,0.66) node[above] {$B$};
 \end{tikzpicture}
 \caption{One iteration of the Douglas--Rachford method for the sets $A = \{x\in\mathbb{E}:\|x\|\leq 1\}$ and $B = \{x\in\mathbb{E}:\langle a,x\rangle=b\}$.}
  \end{center}
\end{figure}
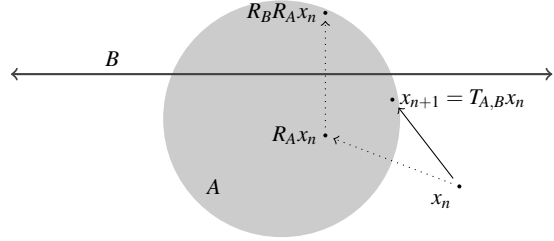

Theorem~\ref{th:convexDR} provides the template for application of the Douglas--Rachford method as a heuristic for non-convex feasibility problems. Furthermore, this theorem also shows that for the Douglas--Rachford method the sequence of primary interest is not the fixed point iterates $(x_n)_{n=1}^\infty$ themselves, but their \emph{shadows} $(P_Ax_n)_{n=1}^\infty$.

\begin{remark}[Douglas--Rachford splitting]
 The Douglas--Rachford reflection method can be viewed as a special case of the \emph{Douglas--Rachford splitting algorithm} for finding a zero of the sum of two maximally monotone operators. This more general splitting method iterates by using \emph{resolvents} of the given maximally monotone operators rather than projection operators of sets. The reflection method is obtained in the special case in which the maximal monotone operators are normal cones to the feasibility problem sets. For details, we refer the reader to \cite{BC10convex}.
\end{remark}

Within an implementation of the Douglas--Rachford method, computation of the projection operators are the component typically requiring the most resources. It is therefore beneficial to store two additional sequences in memory; the \emph{shadow sequence} $(P_Ax_n)_{n=1}^\infty$, and the sequence $(P_BR_Ax_n)_{n=1}^\infty$. This is because iteration \eqref{eq:iterationDR} is expressible as
 \begin{equation}\begin{split}
  x_{n+1} &\in x_n+P_{B}R_{A}x_n-P_{A}x_n\\ &=x_n+P_{B}(2P_{A}x_n-x_n)-P_{A}x_n.
 \end{split}\end{equation}
An implementation utilizing this approach is given in Algorithm~\ref{alg:basicDR}. The stopping criterion uses a relative error and is discussed in Section~\ref{s:results}.

\begin{figure}
\caption{Implementation of the basic Douglas--Rachford algorithm.}\label{alg:basicDR}
\begin{svgraybox}
 \begin{algorithm}[H]
  \KwIn{$x_0\in\mathbb{E}$ and $\epsilon>0$}
  $n=0$\;
  $p_0 \in P_{A}(x_0)$\;
  \While{$n=0${\rm\bf~or~}$\|r_n-p_n\|>\epsilon\|p_n\|$}{
   $r_{n} \in P_{B}(2p_n-x_n)$\;
   $x_{n+1} = x_n+r_n-p_{n}$\;
   $p_{n+1} \in P_{A}(x_{n+1})$\;
   $n = n+1$\;
  }
  \KwOut{$p_n\in\mathbb{E}$}
 \end{algorithm}
\end{svgraybox}
\end{figure}

\section{Protein Conformation Determination}\label{s:protein}
 Proteins are large biomolecules which are comprised of multiple amino acid residues,\footnote{When two amino acids form a peptide bond, a water molecule is formed. An \emph{amino acid residue} is what remains of each amino acid after this reaction.} each of which typically consists of between $10$ and $25$ atoms. 
  Proteins participate is virtually every cellar process, and knowledge of their structural conformation gives insight into the mechanisms by which they perform.

 One of many techniques that can be used to determine conformation is \emph{nuclear magnetic resonance (NMR)}. Currently NMR is only able to non-destructively resolve relatively short distances ({\em i.e.,} those less than $\sim 6$\angstrom). In the proteins we consider, this corresponds to less than $9$\% of all non-zero inter-atom distances.

 We now formulate the problem of protein conformation determination as a computationally tractable matrix completion problem. In fact, our formulation is a \emph{low-rank Euclidean distance matrix completion problem}. We next introduce the necessary definitions.


 We say that a matrix $D=(D_{ij})\in\mathbb{R}^{m\times m}$ is a \emph{Euclidean distance matrix (EDM)} if there exists points $z_1,z_2,\ldots, z_n\in \mathbb{R}^m$ such that
 \begin{equation}\label{eq:EDMentries}
  D_{ij}=\|z_i-z_j\|^2\text{ for }i,j=1,2,\dots,m.
 \end{equation}
Clearly any EDM is symmetric, non-negative, and hollow ({\em i.e.,} contains only zeros along its main diagonal). When~\eqref{eq:EDMentries} holds for a set of points in $\mathbb{R}^q$, we say $D$ is  \emph{embeddable} in $\mathbb{R}^q$. If $D$ is embeddable in $\mathbb{R}^q$ but not in $\mathbb{R}^{q-1}$, then we say that $D$ is \emph{irreducibly embeddable} in $\mathbb{R}^q$.

 We now recall a useful characterization of EDMs, due to Hayden and Wells \cite{HW88approx}. In what follows, the matrix $Q\in\mathbb{R}^{m\times m}$ is the \emph{Householder matrix} given by
   \begin{equation*}
    Q=I-\frac{2vv^T}{v^Tv},\text{ where }v=\begin{bmatrix}1,&1,&\dots &1,&1+\sqrt{m}\end{bmatrix}^T\in\mathbb{R}^m.
   \end{equation*}
\begin{theorem}[EDM characterization {\cite[Th.~3.3]{HW88approx}}]\label{th:EDMchar}
 A non-negative, symmetric, hollow matrix $X\in\mathbb{R}^{m\times m}$ is a Euclidean distance matrix if and only if the block $\widehat{X}\in\mathbb{R}^{(m-1)\times(m-1)}$ in
  \begin{equation}\label{eq:EDMchar}
   Q(-X)Q=\begin{bmatrix}
   \widehat{X} & d \\
   d^T         & \delta \\
  \end{bmatrix}
  \end{equation}
 is positive semi-definite. In this case, $X$ is irreducibly embeddable in $\mathbb{R}^q$ where ${q=\rank(\widehat{X})\leq m-1}$.
\end{theorem}

The problem of \emph{low-rank Euclidean distance matrix completion} can now be formulated. Let $D$ denote a partial Euclidean distance matrix, with entry $D_{ij}$ known whenever $(i,j)\in\Omega$ for some index set $\Omega$, which is embeddable in $\mathbb{R}^q$. Without loss of generality, we make the following three simplifying assumptions on the partial matrix $D$ and index set $\Omega$.
 \begin{enumerate}
  \item (non-negative) $D\geq 0$ ({\em i.e.,} $D_{ij}\geq 0$ for all $i,j=1,2,\dots,m$);
  \item (hollow) $D_{ii}=0$ and $(i,i)\in\Omega$ for $i=1,2,\dots,m$;
  \item (symmetric) $(i,j)\in\Omega\iff (j,i)\in\Omega$, and $D_{ij}=D_{ji}$ for all $(i,j)\in\Omega$.
 \end{enumerate}
We define two constraint sets
\begin{equation}\label{eq:lowRankEDM}\begin{split}
 C_1 &= \left\{X\in S^m:X\geq 0,\,X_{ij}=D_{ij}\text{ for all }(i,j)\in\Omega\right\},\\
 C_2 &= \left\{X\in S^m:Q(-X)Q=\begin{bmatrix}
                                              \widehat{X} & d \\
                                              d^T         & \delta \\
                                            \end{bmatrix},\,
                                            \begin{array}
                                            {ll}\widehat{X}\in S_+^{m-1},& d\in\mathbb{R}^{m-1}\\[0.25ex]  \rank\widehat{X}\leq q,& \delta\in\mathbb{R}\\ \end{array}\right\}.
\end{split}\end{equation}
In light of Theorem~\ref{th:EDMchar}, the problem of \emph{low-rank Euclidean distance matrix completion} can be cast as the two-set feasibility problem
 \begin{equation*}
  \text{find~}X\in C_1\cap C_2.
 \end{equation*}
That is, a matrix $X$ is a low-rank Euclidean distance matrix which completes $D$ if and only if $X\in C_1\cap C_2$. Some comments regarding the constraint sets in \eqref{eq:lowRankEDM} are in order.

 The set $C_1$ encodes the experimental data obtained from NMR, and the \emph{a priori} knowledge that the matrix is non-negative, symmetric and hollow. Its projection has a simple formulae, as we now show.
 \begin{proposition}[Projection onto $C_1$]\label{prop:PC1}
  Let $X\in\mathbb{R}^{m\times m}$. Then $P_{C_1}X$ is given element-wise by
   \begin{equation*}
    (P_{C_1}X)_{ij} = \begin{cases}
                        D_{ij}, & (i,j)\in\Omega \\
                        \max\{0,X_{ij}\}, & (i,j)\not\in\Omega \\
                      \end{cases}\quad\text{for}\quad i,j=1,2,\dots,m.
   \end{equation*}
 \end{proposition}
 \begin{proof}
  Let $Y$ be any matrix in $C_1$. We have
   \begin{equation}\label{e:PC1-1}
   \begin{split}
     \|X-Y\|^2_F
    &= \sum_{(i,j)\in\Omega}(X_{ij}-Y_{ij})^2+\sum_{\substack{(i,j)\not\in\Omega\\ \text{s.t.~}X_{ij}<0}}(X_{ij}-Y_{ij})^2 + \sum_{\substack{(i,j)\not\in\Omega\\ \text{s.t.~}X_{ij}\geq 0}}(X_{ij}-Y_{ij})^2 \\
 &= \sum_{(i,j)\in\Omega}(X_{ij}-D_{ij})^2 +\sum_{\substack{(i,j)\not\in\Omega\\ \text{s.t.~}X_{ij}<0}}X_{ij}^2 + \sum_{\substack{(i,j)\not\in\Omega\\ \text{s.t.~}X_{ij}\geq 0}}(X_{ij}-Y_{ij})^2.
   \end{split}
   \end{equation}
  Let $P$ be the matrix given by the proposed projection formula (clearly $P\in C_1$). Then
   \begin{equation}\label{e:PC1-2}
    \sum_{\substack{(i,j)\not\in\Omega\\ \text{s.t.~}X_{ij}\geq 0}}(X_{ij}-Y_{ij})^2\geq \sum_{\substack{(i,j)\not\in\Omega\\ \text{s.t.~}X_{ij}\geq 0}}(X_{ij}-X_{ij})^2=\sum_{\substack{(i,j)\not\in\Omega\\ \text{s.t.~}X_{ij}\geq 0}}(X_{ij}-P_{ij})^2.
   \end{equation}
 By combining \eqref{e:PC1-1} and \eqref{e:PC1-2} we see that
  $$\|X-Y\|^2_F\geq \|X-P\|^2_F\text{ for all }Y\in C_1.$$
 Since $C_1$ is closed and convex, $P$ is the unique nearest point to $X$ in $C_1$.
 \qed\end{proof}

 \begin{remark}
  Since $C_1$ is a closed convex set, an alternative (less direct) proof of Proposition~\ref{prop:PC1} can be given using the standard variational characterization of convex projections \cite[Th.~1.2.4]{C12iterative}.
 \end{remark}

 Using the necessary condition given by Theorem~\ref{th:EDMchar}, the non-convex set $C_2$ encodes the \emph{a priori} knowledge that the matrix of interest is a EDM together with the dimension of the space in which the corresponding points generating the matrix are contained. We now derive the projection onto $C_2$.

\begin{theorem}[Nearest low-rank EDMs {\cite{ABT14matrix}}]\label{th:PC2}
  Let $X\in S^m$ be a non-negative, hollow matrix. Then
  $$P_{C_2}(X)=\left\{-Q\begin{bmatrix}
  \widehat Y & d \\
   d^T         & \delta \\
  \end{bmatrix}Q :
   Q(-X)Q=\begin{bmatrix}
   \widehat{X} & d \\
   d^T         & \delta \\
  \end{bmatrix},\,\begin{array}{l}\widehat{X}\in\mathbb{R}^{(m-1)\times(m-1)},\\  d\in\mathbb{R}^{m-1},\hfill \delta\in\mathbb{R},\\ \end{array}\;\widehat{Y}\in P_{M}\widehat X\right\},$$
 where $M$ is the set of positive semi-definite matrices with rank $q$ or less. In particular, $P_{C_2}(X)$ is a singleton if and only if $P_M\widehat X$ is a singleton.
\end{theorem}
\begin{proof}
  Let $Y$ be any matrix in $C_2$. That is,
   $$Y=\begin{bmatrix}
          \widehat{Y} & c \\
          c^T         & \beta \\
         \end{bmatrix},\quad\text{for some }c\in\mathbb{R}^{m-1},\,\beta\in\mathbb{R},\,\widehat{Y}\in S.$$
Using the orthogonality of $Q$, we compute
   \begin{equation}\label{e:t-PC2}
   \begin{split}
    \|X-Y\|^2_F
    &= \|Q(X-Y)Q\|^2_F = \|Q(-X)Q-Q(-Y)Q\|^2_F \\
    &= \left\|
         \begin{bmatrix}
          \widehat{X} & d \\
          d^T         & \delta \\
         \end{bmatrix} -
         \begin{bmatrix}
          \widehat{Y} & c \\
          c^T         & \beta \\
         \end{bmatrix}
       \right\|^2_F
     = \left\|
         \begin{bmatrix}
          \widehat{X}-\widehat{Y} & (d-c) \\
          (d-c)^T         & (\delta-\beta) \\
         \end{bmatrix}
       \right\|^2_F \\
    &= \|\widehat{X}-\widehat{Y}\|^2_F+2\|d-c\|^2+|\gamma-\beta|^2.
   \end{split}
   \end{equation}
  To complete the proof we observe that \eqref{e:t-PC2} is minimized if and only if $c=d,\gamma=\beta$ and $\widehat{Y}\in P_M\widehat{X}$.
\qed\end{proof}

The set $M$ in Theorem~\ref{th:PC2} is a set of low-rank positive semi-definite matrices. 
One method to compute its projection (and the one we will use) is by exploiting the eigen-decomposition of $\widehat{X}$. Denote by $\diag(\lambda)$ the diagonal matrix given by placing the elements of the vector $\lambda\in\mathbb{R}^m$ along the main diagonal. Let $\widehat{X}=U\diag(\lambda)U^T$ be an eigen-decomposition (of $\widehat{X}$) with
 $$\lambda_1\geq \lambda_2\geq \dots\geq \lambda_q^+\geq \dots\geq \lambda_m.$$
A projection onto the set is then given by
 $$U\diag((\lambda_1^+,\lambda_2^+,\dots,\lambda_q^+,0\dots,0,0))U^T,$$
where $x^+$ denotes $\max\{0,x\}$.

\section{Computational Experiments}\label{s:results}
We apply the formulation of Section~\ref{s:protein} to six proteins, shown in Table~\ref{tab:proteinInfo}, obtained from the \emph{RCSB Protein Data Bank}\footnote{RCSB Protein Data Bank:~\url{www.rcsb.org/pdb}}. As part of \cite{ABT14matrix}, reconstructions of the same six proteins were attempted using a partial EDM containing only distances less than $6$\angstrom. Here we attempt reconstructions using partial EDMs which, in addition to these short-range distances, incorporate other \emph{a priori} information. In particular, we include  inter-atomic distances greater than $6$\angstrom{} for atoms from within the same residue in the partial EDM. This is reasonable since the structure of the individual residues is known. For 1PTQ, this information gives approximately a further $0.2\%$ of the total non-zero inter-atomic distances.

\begin{table}
\caption{Number of atoms, residues, known, and total non-zero inter-atomic distances in our six test proteins.}\label{tab:proteinInfo}
\begin{tabular*}{\linewidth}{@{\extracolsep{\fill}}ccccc}
  \hline\noalign{\smallskip}
  Protein  & Atoms  & Residues  & Total Non-Zero Distances  & Known Non-Zero Distances \\
  \noalign{\smallskip}\hline\noalign{\smallskip}
  1PTQ     & \hspace{1ex}404              & \hspace{1ex}50  & \hspace{1ex}81,406    & 8.9207\% \\
  1HOE     & \hspace{1ex}581              & \hspace{1ex}74  & 168,490               & 6.4105\%  \\
  1LFB     & \hspace{1ex}641              & \hspace{1ex}99  & 205,120               & 5.6362\%  \\
  1PHT     & \hspace{1ex}988              & \hspace{1ex}85  & 236,328               & 4.6501\%  \\
  1POA     & 1067                         & 118             & 568,711               & 3.6375\%  \\
  1AX8     & 1074                         & 146             & 576,201               & 3.5606\%  \\
  \noalign{\smallskip}\hline\noalign{\smallskip}
\end{tabular*}
\end{table}

Our experiments were implemented in \emph{Cython} and performed on a machine having an Intel Xeon E5540 $@$ $2.83$GHz running Red Hat Enterprise Linux 6.5. A combination of the Cython platform, and optimized code gave approximately a ten-fold speed up compared to \cite{ABT14matrix}. This allowed for a greater number of iterations to be performed and hence the use of the more robust (albeit still heuristic) stopping criterion given in Algorithm~\ref{alg:basicDR} as opposed to simply performing a fixed number of iterations. The reconstructed EDM, $x$, was converted to points $z_1,z_2,\dots,z_m\in\mathbb{R}^3$ using Algorithm~\ref{alg:EDM2pts}.

\begin{figure}
\caption{Conversion of EDM to points in $\mathbb{R}^q$.}\label{alg:EDM2pts}
\begin{svgraybox}
 \begin{algorithm}[H]
  \KwIn{$x\in X$ \tcc*{a Euclidean distance matrix}}
  $L=I-ee^T/n$ where $e=(1,1,\dots,1)^T$\;
  $\tau = -LDL/2$\;
  $USV^T=\operatorname{SingularValueDecomposition}{(\tau)}$\;
  $Z=$ first $q$ columns of $U\sqrt{S}$\;
  $z_i=i$th row of $Z$ for $i=1,2,\dots,m$\;
  \KwOut{$z_1,z_2,\dots,z_q\in\mathbb{R}^q$ \tcc*{points corresponding to $x$}}
 \end{algorithm}
\end{svgraybox}
\end{figure}

\begin{remark}
 It is worth emphasing that our primary concern is the quality of the reconstruction, rather than the time required to perform the reconstruction. This is because, if done well, one only needs to determine the conformation once.
\end{remark}

 We report two error metrics, which we now explain. Denote the actual EDM by $x^{\text{actual}}$. The first error metric is a measure of the \emph{error in the reconstructed EDM}, and is given by
 \begin{equation*}
  \text{EDM-error}=\|x^{\text{actual}}-x\|_F=\sqrt{\sum_{i,j=1}^m\left|x^{\text{actual}}_{ij}-x_{ij}\right|^2}.
 \end{equation*}
 
 Denote the actual atom positions by $z^{\text{actual}}_1,z^{\text{actual}}_2,\dots,z^{\text{actual}}_m\in\mathbb{R}^3$. The second error metric measures the \emph{error in the reconstructed atom positions} $z_1,z_2,\dots,z_m\in\mathbb{R}^3$. Since EDMs are invariant under translation, reflection, and rotation of the points by which they are induced, we first perform a \emph{Procrustes analysis} \cite{D05} to obtain $\widetilde{z_1},\widetilde{z_2},\dots,\widetilde{z_m}\in\mathbb{R}^3$. These points are a best fit of the reconstructed points when the aforementioned transformations are allowed.
 The second error metric is  given by
 \begin{equation*}
  \text{Position-error}=\sqrt{\sum_{k=1}^m\|z^{\text{actual}}_k-\widetilde{z_k}\|^2_2}.
 \end{equation*}
When comparing the relative size of these two errors, it is worth noting that the summation in the EDM-error contains $m^2$ terms whereas the summation in the position-error contains only $3m$.

\begin{remark}[Decibel error]
 It is also common to consider the relative error in \emph{decibels (dB)}, as was reported in \cite{ABT14matrix}. That is,
  $$\text{Relative error (dB)}=10\log_{10}\left(\frac{\|P_BR_Ax-P_Ax\|_F^2}{\|P_Ax\|_F^2}\right).$$
 In this study the relative error in decibels is not reported. This is unnecessary because the stopping criterion used in Algorithm~\ref{alg:basicDR} is equivalent to requiring that the decibel error be less than $10\log_{10}(\epsilon^2)$. Requiring that $\epsilon=10^{-5}$ corresponds to aiming at a relative error of $-100$dB.
\end{remark}

\begin{remark}[Stopping criterion and tolerance]\label{rem:stopping}
 In the computational experiments that follow, the stopping tolerance is taken to be $\epsilon=10^{-5}$. We now provide some justification for this choice.
   \begin{figure}
    \centering
    \includegraphics[width=0.7\linewidth]{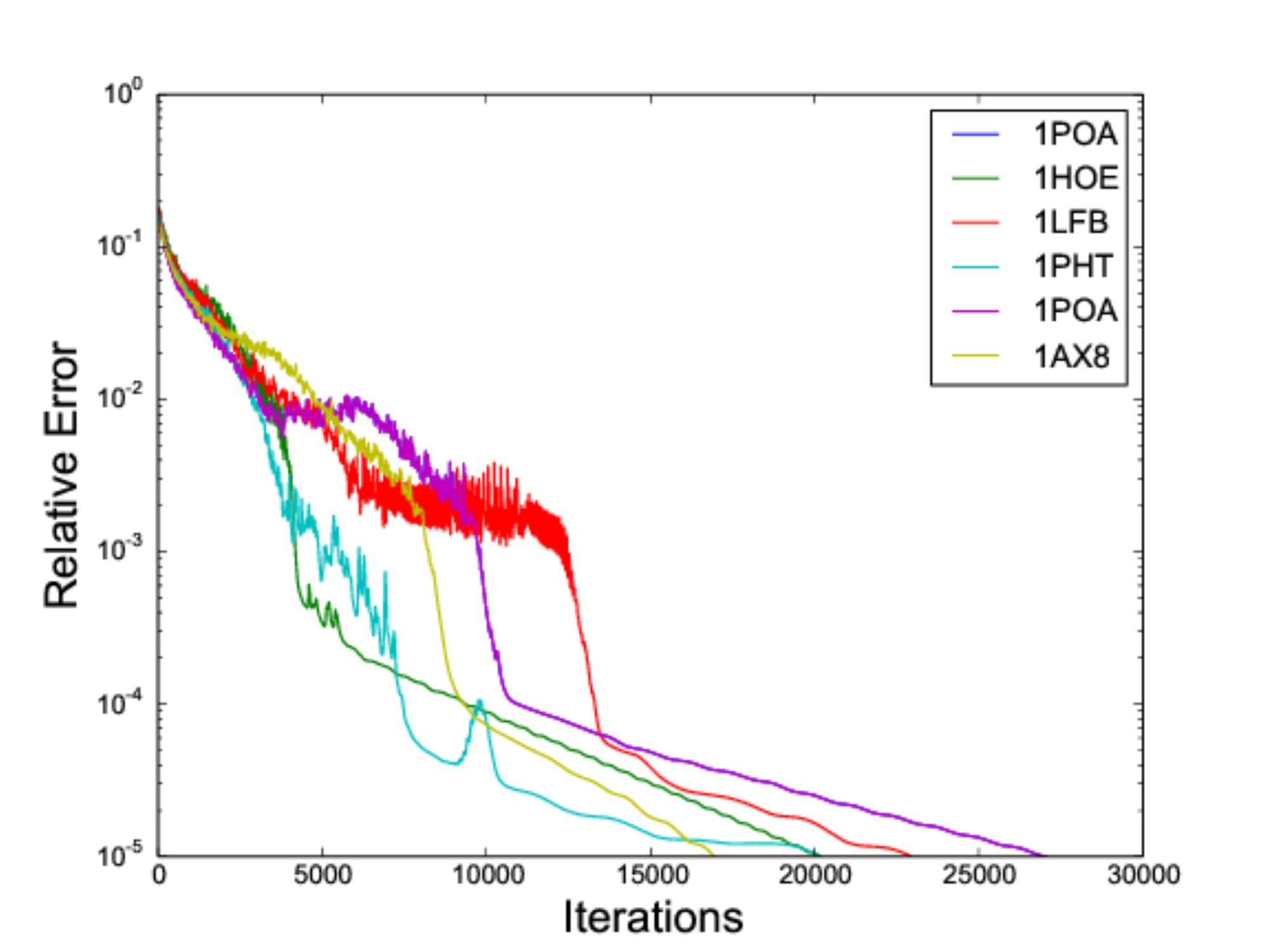}
    \caption{The relative error as a function of iterations (vertical axis is logarithmic).}\label{fig:stopping}
\end{figure}

   For each of the six proteins, Figure~\ref{fig:stopping} shows the relative error as a function of the number of iterations starting from a given initial point for the Douglas--Rachford method. 
   
   \begin{itemize}
     \item   When the number of iteration is less than $5000$ the relative error exhibits non-monotone oscillatory behaviour --- which seems to provide much of the potency of the method. It seems to allow the reflection method to sample regions and avoid settling at an inferior local minimum of the configuration space. In \cite{ABT14matrix} we observed that the alternating projection method, which is monotonic, fails to produce good reconstructions.
 
     \item    When the relative error is between $10^{-3}$ and $10^{-4}$, it decreases sharply after which a period of more predictable decrease is observed. 
         
 \item  Beyond this point slower progress is made. We therefore choose our stopping tolerance to be $\epsilon=10^{-5}$ so that the algorithm will terminate in this region.
   \end{itemize}

  The change in successive iterates was found to also exhibit similar behavior (not shown), so is  another suitable candidate for a stopping criterion.

\end{remark}

It is worth noting that there are many other techniques for solving (variants of) the protein conformation problem (see for instance \cite{BetaMDGP}).  Such a discussion, however, is beyond the scope of this chapter.

\subsection{Basic Douglas--Rachford Algorithm Results}\label{ss:resultsBasic}
Table~\ref{tab:basicDRrecon} gives results for the basic Douglas--Rachford algorithm presented in Algorithm~\ref{alg:basicDR}. We make some comments regarding these results.

The EDM-error increases with increasing problem size; yet the same trend is not observed for the position-error for which 1PHT reported the largest error. For all of the proteins studied, the differences between the average and worst case results for the position-errors were small. This strongly suggests that the method can consistently produce a EDM which gives the desired atomic positions.

The second column of Figure~\ref{fig:recons} shows the conformation of the basic Douglas--Rachford reconstructions, which are visually indistinguishable from the actual conformation shown in the first column. This is an improvement from what was reported in \cite{ABT14matrix} whose Douglas--Rachford reconstructions of two of the larger proteins, 1POA and 1AX8, gave unrealistic conformations consisting of disjoint blocks of atoms. In light of Remark~\ref{rem:stopping} it is likely that this was due to premature algorithm termination.

\begin{table}
\caption{Average (worst) results from five random replications of the basic Douglas--Rachford algorithm with $\epsilon=10^{-5}$.}\label{tab:basicDRrecon}
\begin{tabular*}{\linewidth}{@{\extracolsep{\fill}}lp{0.2cm}*{2}{d{4}d{5}p{0.2cm}}d{1}d{0}p{0.2cm}d{2}d{3}}
  \hline\noalign{\smallskip}
  Protein  && \multicolumn{2}{c}{EDM-Error} && \multicolumn{2}{c}{Position-Error} && \multicolumn{2}{c}{Iterations} && \multicolumn{2}{c}{Time (h)} \\
  \noalign{\smallskip}\hline\noalign{\smallskip}
  1PTQ     &&  3.6816 & (4.0938)  &&  0.1307 &  (0.1457) &&  4339.6 &  (4686) &&  0.28 &  (0.30) \\
  1HOE     &&  9.7475 & (13.8503) &&  0.1781 &  (0.2636) && 20794.4 & (21776) &&  3.50 &  (3.67) \\
  1LFB     &&  9.8728 & (17.2860) &&  1.1388 &  (2.1177) && 22346.2 & (23295) &&  4.64 &  (4.85) \\
  1PHT     && 10.3709 & (12.9557) && 12.8782 & (13.0056) && 20103.0 & (20251) && 13.90 & (14.00) \\
  1POA     && 25.4225 & (46.5804) &&  0.5844 &  (1.1639) && 28426.0 & (29766) && 23.33 & (24.47) \\
  1AX8     && 25.7369 & (39.4586) &&  0.6592 &  (0.9160) && 17969.8 & (19059) && 15.04 & (15.95) \\
  \noalign{\smallskip}\hline\noalign{\smallskip}
\end{tabular*}
\end{table}

\begin{figure}
 \centering
 \begin{tabular}{p{0.02\linewidth}*{4}{m{0.21\linewidth}}}
  & \multicolumn{1}{c}{Actual} & \multicolumn{1}{c}{Section~\ref{ss:resultsBasic}} & \multicolumn{1}{c}{Section~\ref{ss:resultsLazy}} & \multicolumn{1}{c}{Section~\ref{ss:resultsTen}} \\[1ex]
  \rotatebox[origin=c]{90}{1PTQ}
  & \includegraphics[width=\linewidth]{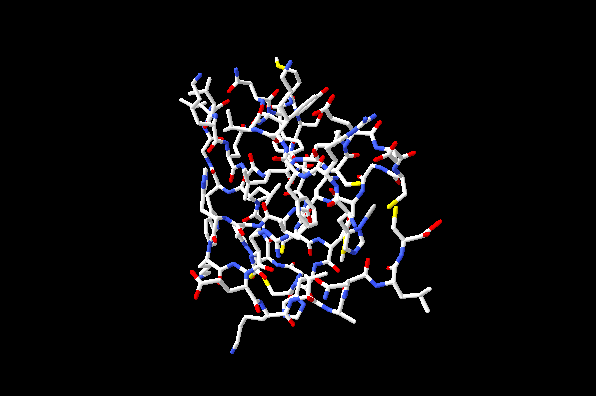}
  &  \includegraphics[width=\linewidth]{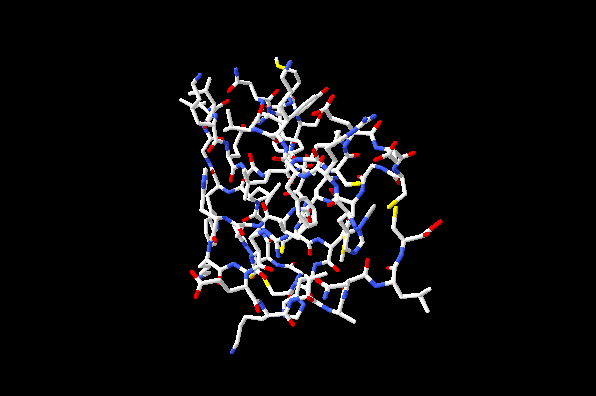}
  &  \includegraphics[width=\linewidth]{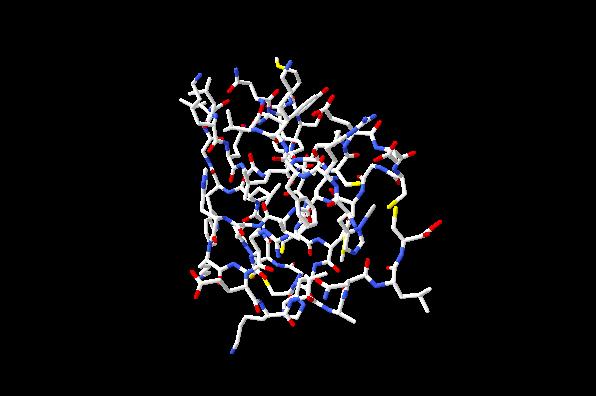}
  &  \includegraphics[width=\linewidth]{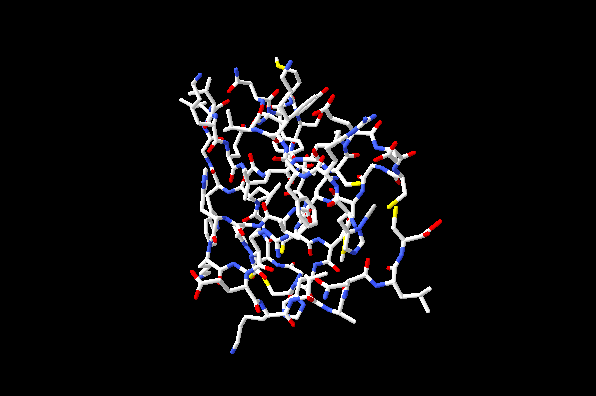}  \\
  \rotatebox[origin=c]{90}{1HOE}
  & \includegraphics[width=\linewidth]{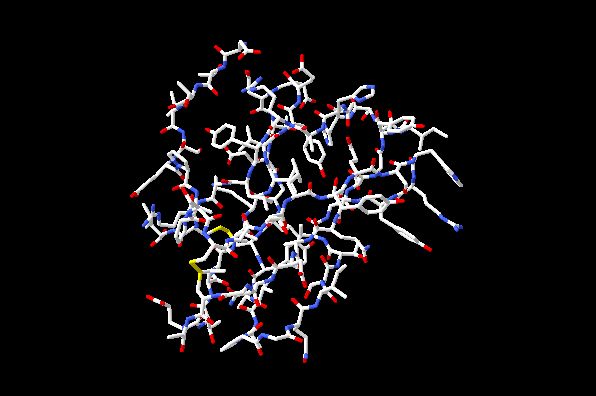}
  &  \includegraphics[width=\linewidth]{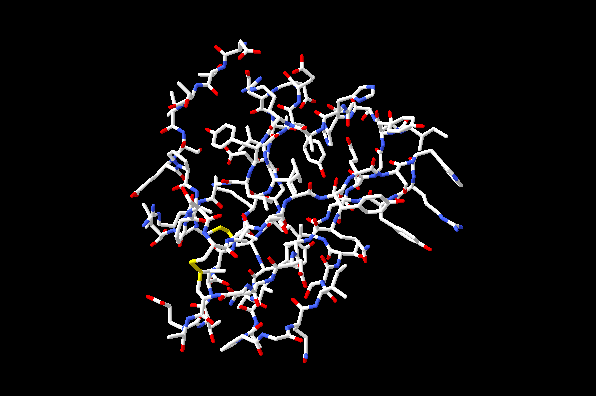}
  &  \includegraphics[width=\linewidth]{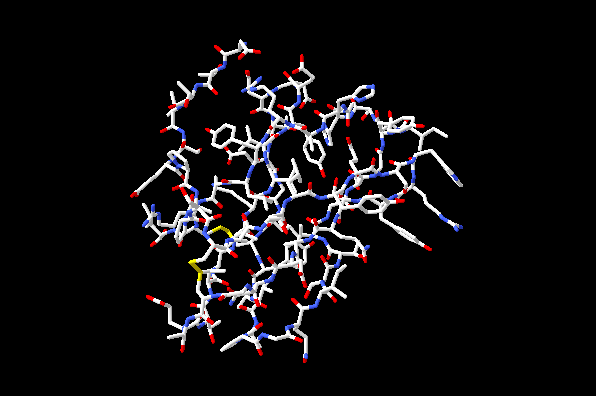}
  &  \includegraphics[width=\linewidth]{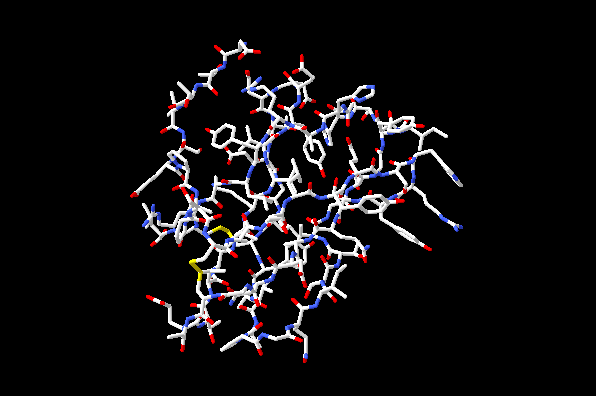}  \\
  \rotatebox[origin=c]{90}{1LFB}
  & \includegraphics[width=\linewidth]{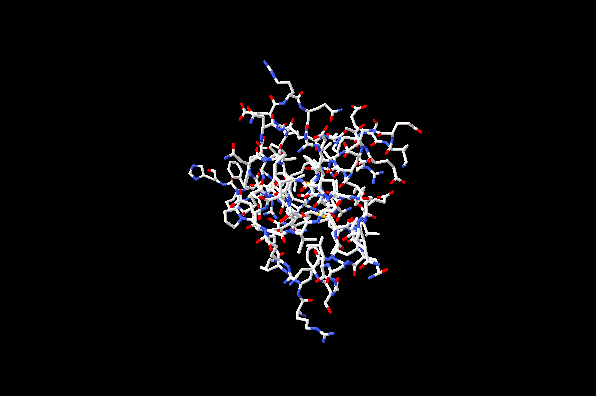}
  &  \includegraphics[width=\linewidth]{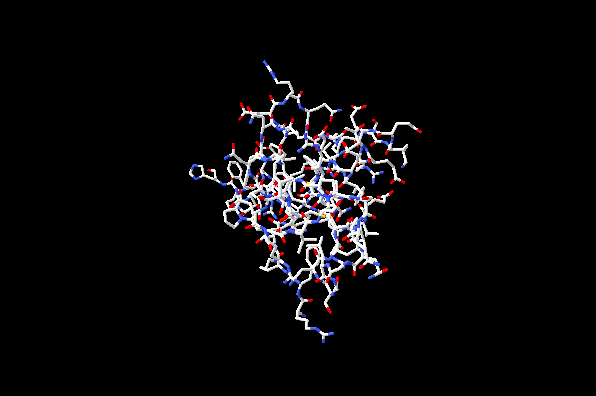}
  & \includegraphics[width=\linewidth]{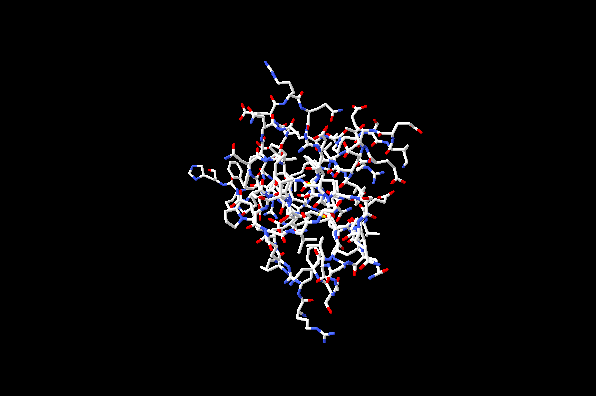}
  &  \includegraphics[width=\linewidth]{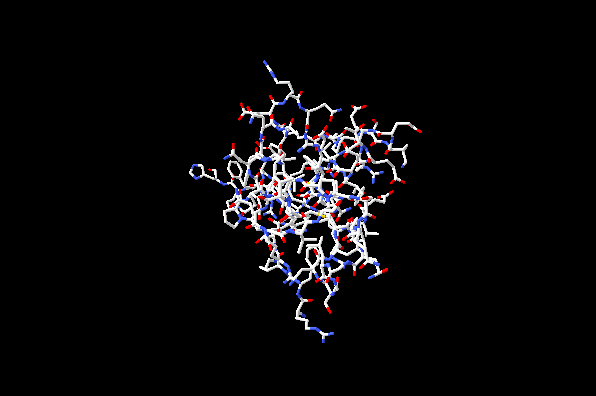}  \\
  \rotatebox[origin=c]{90}{1PHT}
  & \includegraphics[width=\linewidth]{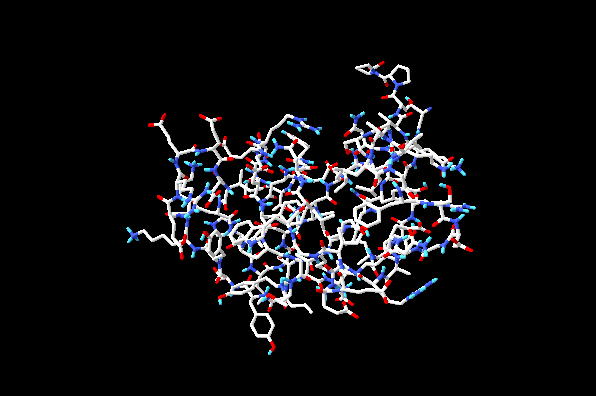}
  &  \includegraphics[width=\linewidth]{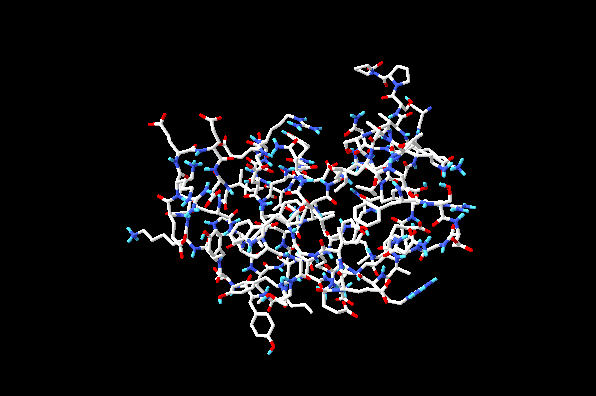}
  &  \includegraphics[width=\linewidth]{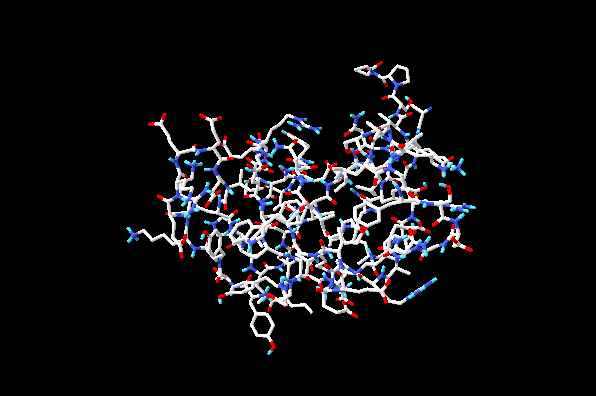}
  &  \includegraphics[width=\linewidth]{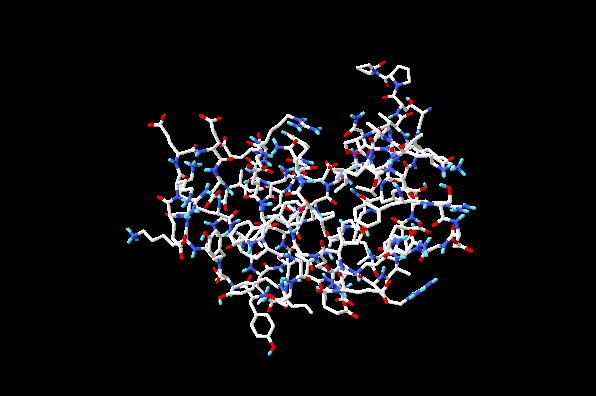}  \\
  \rotatebox[origin=c]{90}{1POA}
  & \includegraphics[width=\linewidth]{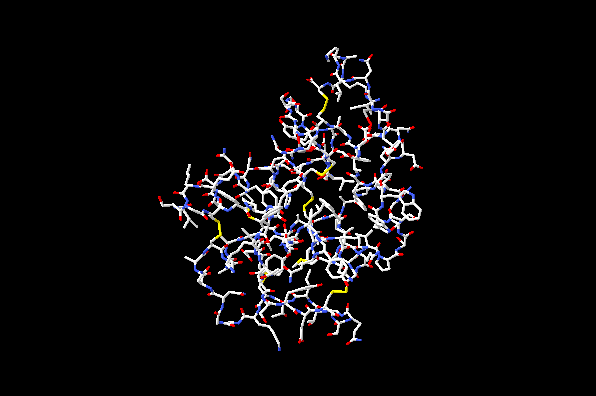}
  &  \includegraphics[width=\linewidth]{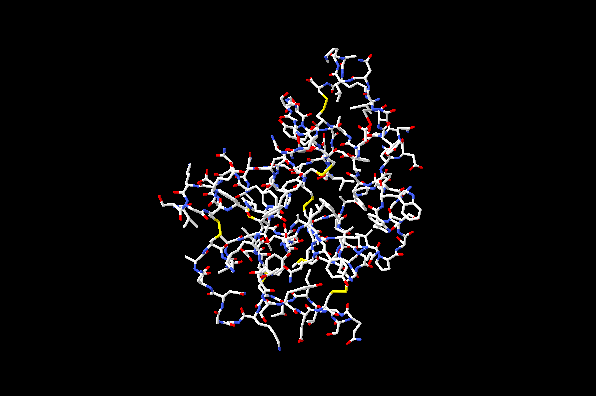}
  &  \includegraphics[width=\linewidth]{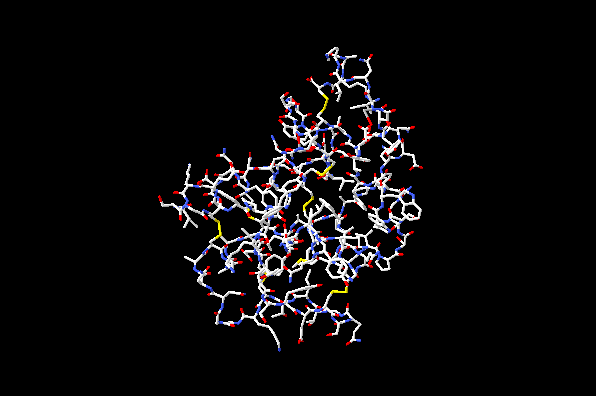}
  &  \includegraphics[width=\linewidth]{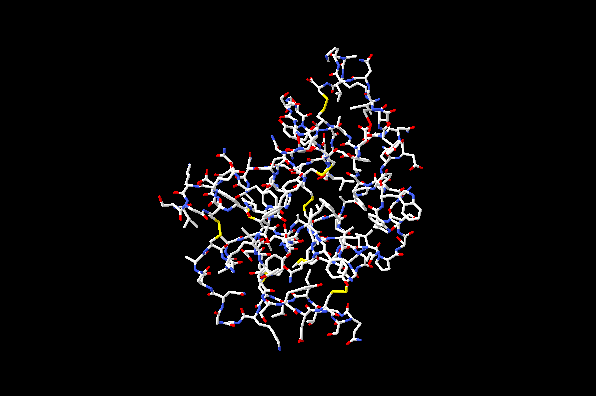}  \\
  \rotatebox[origin=c]{90}{1AX8}
  & \includegraphics[width=\linewidth]{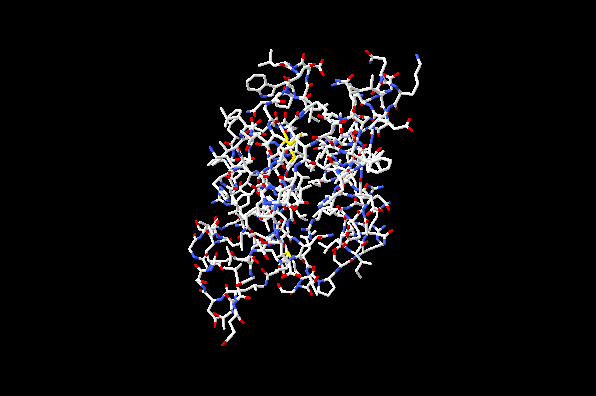}
  &  \includegraphics[width=\linewidth]{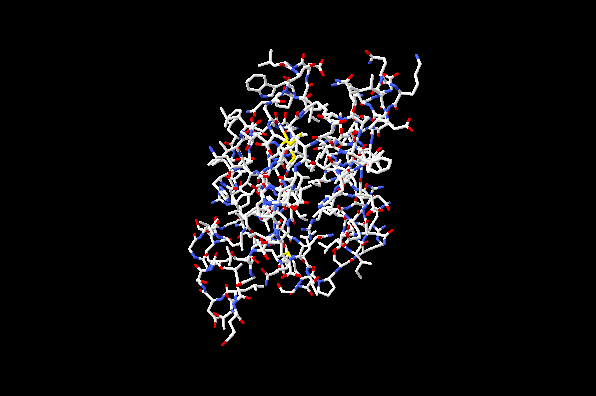}
  &  \includegraphics[width=\linewidth]{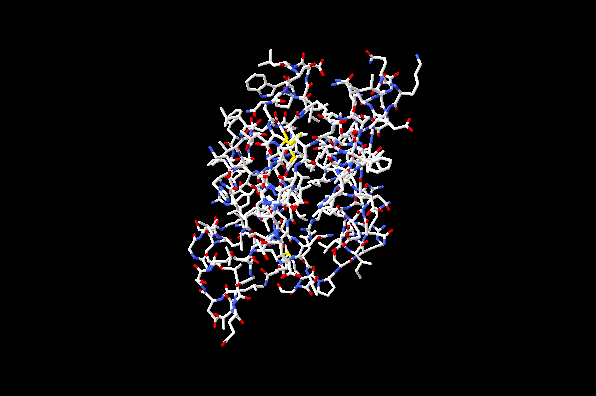}
  &  \includegraphics[width=\linewidth]{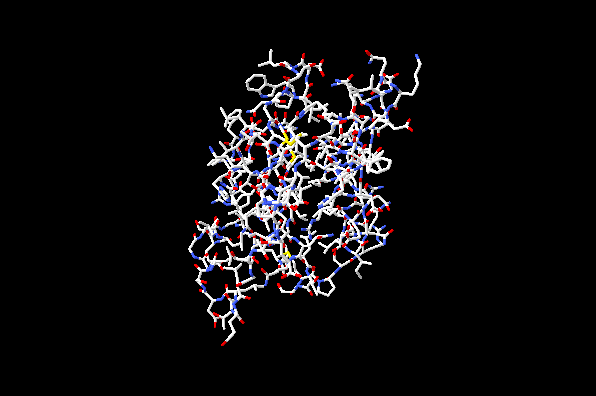}  \\
 \end{tabular}
 \caption{The conformations of the six proteins, and their three Douglas--Rachford reconstructions.}\label{fig:recons}
\end{figure}

\subsection{Douglas--Rachford Algorithm with Periodic Rank Projections}\label{ss:resultsLazy}
In our formulation of the protein confirmation problem, the most expensive step is the computation of the projection onto the rank constraint $C_2$. Thus requires the eigen-decomposition of a $(m-1)\times(m-1)$ symmetric matrix. In this section we propose problem specific heuristics which allow for this computation to sometimes be avoided.

 One idea to avoid performing the eigen-decomposition is to not update the sequence $(r_n)_{n=1}^\infty$ in Algorithm~\ref{alg:basicDR} at every iteration but only periodically. This approach is described in Algorithm~\ref{alg:periodicPC2}, and results, with updates only every third time, in Table~\ref{tab:lazyResults}.

 We now compare the results of this section to those of Section~\ref{ss:resultsBasic}. A small increase in the position-errors, and a larger increase in the EDM-errors was observed. The number of iterations required also increased, with this number almost doubling for 1PTQ. For all six test proteins, the total time required was less. The biggest improvement was 1POA whose total time was more than halved. The quality of the reconstructed conformations seem not to be adversely effected by the use of periodic rank projections, as can be seen in Figure~\ref{fig:recons}.

\begin{figure}
\caption{The Douglas--Rachford algorithm with $T$-periodic projections onto the set $B$.}\label{alg:lazyPC2}
\begin{svgraybox}
 \begin{algorithm}[H]\label{alg:periodicPC2}
  \KwIn{$x_0\in X, T\in\mathbb{N}$ and $\epsilon>0$}
  $n=0$\;
  $p_0 \in P_{A}(x_0)$\;
  \While{$n=0${\rm\bf~or~}$\|r_n-p_n\|>\epsilon\|p_n\|$}{
   \eIf{$n\bmod{T}=0$}{
     $r_n \in P_{B}(2p_n-x_n)$\;
   }{
     $r_n = r_{n-1}$\;
   }
   $x_{n+1} = x_n+r_n-p_{n}$\;
   $p_{n+1} \in P_{A}(x_{n+1})$\;
   $n = n+1$\;
  }
  \KwOut{$p_n\in X$}
 \end{algorithm}
\end{svgraybox}
\end{figure}

\begin{table}
\caption{Average (worst) results from five random replications of the Douglas--Rachford algorithm with periodic rank projections with $T=3$ and $\epsilon=10^{-5}$.}\label{tab:lazyResults}
\begin{tabular*}{\linewidth}{@{\extracolsep{\fill}}lp{0.2cm}*{2}{d{4}d{5}p{0.2cm}}d{1}d{0}p{0.2cm}d{2}d{3}}
  \hline\noalign{\smallskip}
  Protein  && \multicolumn{2}{c}{EDM-Error} && \multicolumn{2}{c}{Position-Error} && \multicolumn{2}{c}{Iterations} && \multicolumn{2}{c}{Time (h)} \\
  \noalign{\smallskip}\hline\noalign{\smallskip}
  1PTQ     &&  4.3709 &  (4.7200) &&  0.1919 &  (0.2240) &&  7160.6 &  (7595) && 0.16 & (0.17) \\
  1HOE     &&  10.1790 &  (12.1089) &&  0.2603 &  (0.2933) && 20305.4 & (22550) &&  1.21 & (1.35)  \\
  1LFB     &&  17.6532 &  (19.0984) &&  1.2709 &  (1.7243) && 28983.8 & (31211) &&  2.15 & (2.31)  \\
  1PHT     &&  23.8594 &  (25.9794) && 13.1358 & (13.2805) && 20559.2 & (20981) &&  5.03 & (5.13) \\
  1POA     &&  49.8406 &  (51.3411) &&  1.0948 &  (1.2084) && 33150.8 & (39083) &&  9.55 & (11.25)\\
  1AX8     &&  45.5203 &  (49.1866) &&  1.1696 &  (1.4482) && 27080.6 & (31250) &&  7.96 & (9.20) \\
  \noalign{\smallskip}\hline\noalign{\smallskip}
\end{tabular*}
\end{table}

\subsection{Reconstructions with Additional Distance Data}\label{ss:resultsTen}
In Sections~\ref{ss:resultsBasic} \& \ref{ss:resultsLazy} we considered the physically realistic setting in which distances below the threshold of $6$\angstrom{} were known. As noted, when the number of atoms in a protein increases, the proportion of inter-atomic distances below this threshold compared to the total number of (non-zero) distances decreases. 

To better understand the Douglas--Rachford method applied to larger problem instances, we performed the same reconstruction as in Section~\ref{ss:resultsBasic} but with the percentage of known non-zero distances constant. More precisely, we assumed that the smallest $10\%$ of inter-atomic distances were known.

\begin{table}
\caption{Average (worst) results from five random replications of the basic Douglas--Rachford algorithm from the smallest $10\%$ of inter-atomic distances with $\epsilon=10^{-5}$.} \label{tab:additionalDRrecon}
\begin{tabular*}{\linewidth}{@{\extracolsep{\fill}}lp{0.2cm}*{2}{d{4}d{5}p{0.2cm}}d{1}d{0}p{0.2cm}d{2}d{3}}
  \hline\noalign{\smallskip}
  Protein  && \multicolumn{2}{c}{EDM-Error} && \multicolumn{2}{c}{Position-Error} && \multicolumn{2}{c}{Iterations} && \multicolumn{2}{c}{Time (h)} \\
  \noalign{\smallskip}\hline\noalign{\smallskip}
  1PTQ     &&  3.1924 &  (3.5936) &&  0.0963 &  (0.1213) &&  4014.4 &  (4184) && 0.26 & (0.27) \\
  1HOE     &&  8.0905 & (10.4357) &&  0.0960 &  (0.1265) && 15110.4 & (15709) && 2.54 & (2.64) \\
  1LFB     &&  7.2941 & (13.9893) &&  0.4647 &  (0.9182) && 11060.6 & (11912) && 2.29 & (2.46) \\
  1PHT     && 14.1302 & (20.2476) &&  0.3542 &  (0.4326) &&  6071.0 &  (6512) && 4.19 & (4.49) \\
  1POA     && 19.5619 & (31.1987) &&  0.1624 &  (0.2665) && 11555.8 & (13244) && 9.44 & (10.81) \\
  1AX8     && 14.0747 & (29.7259) &&  0.0940 &  (0.1922) && 10099.2 & (11125) && 8.38 & (9.23) \\
  \noalign{\smallskip}\hline\noalign{\smallskip}
\end{tabular*}
\end{table}

As could perhaps be predicted, when more distance information is incorporated the error metrics, and the number of iterations decrease. Problem size and EDM-error do not correlate as strongly compared to the results of Section~\ref{ss:resultsBasic}. However, the general trend that larger problem sizes give larger EDM-errors is still observed. The most notable improvement, when compared to Section~\ref{ss:resultsBasic}, is the position-error for 1PHT. This suggests that in the realistic setting of Section~\ref{ss:resultsBasic} the underlying protein's conformation ({\em e.g.,} a compact or a dispersed conformation) is an important factor in the difficulty of the reconstruction problem.

\subsection{Ionic Liquid Bulk Structure Determination}
Ionic liquids (ILs) are salts ({\em i.e.,} they are comprised of positively and negatively charged ions) having low melting points, typically occupying the liquid state at room temperature. An analogous  reconstruction problem arising in the context of ionic liquid chemistry is to determine a given ionic liquid's \emph{bulk structure}. That is, the configuration of its ions with respect to each other (the structure of the individual ions is known).

In this section, we applied the Douglas--Rachford method to a simplified version of this problem. Entries of the partial EDM are assumed to be known whenever the two atoms are bonded ({\em i.e.,} when their \emph{Van der Waals radii} taken from \cite{vdw} overlap).

Table~\ref{tab:ILs} reports results for a \emph{propylammonium nitrate (PAN)} data set consisting of 180 atoms. The corresponding rank-$3$ EDM completion problem has a total of 32,220 non-zero inter-atomic distances of which 5.95\% form the partial EDM.

\begin{table}
\caption{Average (worst) results from five random replications of the basic Douglas--Rachford algorithm, applied to ionic liquid bulk structure determination, with $\epsilon=10^{-5}$.}\label{tab:ILs}
\begin{tabular*}{\linewidth}{@{\extracolsep{\fill}}p{0cm}*{2}{d{4}d{5}p{0.2cm}}d{1}d{0}p{0.2cm}d{2}d{3}p{0cm}}
  \hline\noalign{\smallskip}
  &\multicolumn{2}{c}{EDM-Error} && \multicolumn{2}{c}{Position-Error} && \multicolumn{2}{c}{Iterations} && \multicolumn{2}{c}{Time (h)} & \\ 
  \noalign{\smallskip}\hline\noalign{\smallskip}
 & 0.6323 & (0.6918) && 2.0374 & (2.5039) && 41553.2 & (82062) && 0.22 & (0.43) & \\
  \noalign{\smallskip}\hline\noalign{\smallskip}
\end{tabular*}
\end{table}

As was the case in the protein conformation application, the difference between the average and worst case results for the two error metrics is observed to be small. The actual conformation of PAN, and its Douglas--Rachford reconstruction are shown in Figure~\ref{fig:ILrecon}. A high degree of visual coincidence is observed, although a small amount of the finer detail is missing.

\begin{figure}
 \centering
 \includegraphics[width=0.4\linewidth]{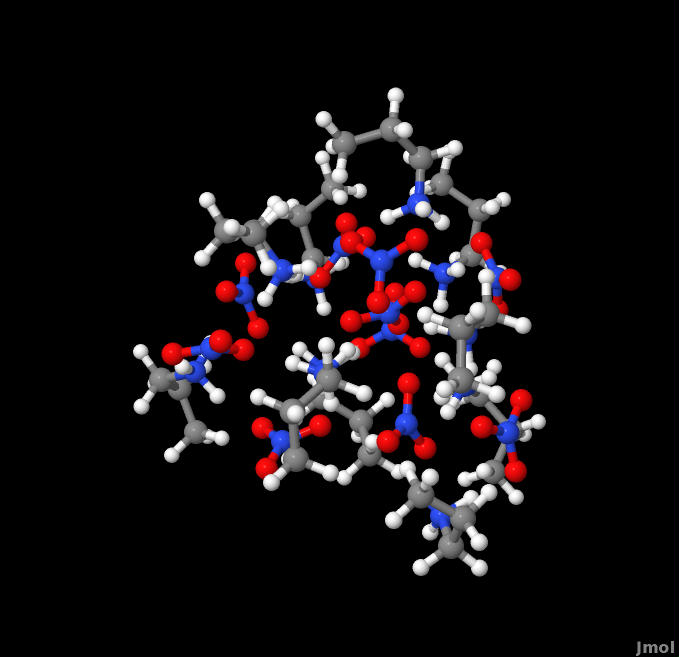} \hspace{0.07\linewidth}
 \includegraphics[width=0.4\linewidth]{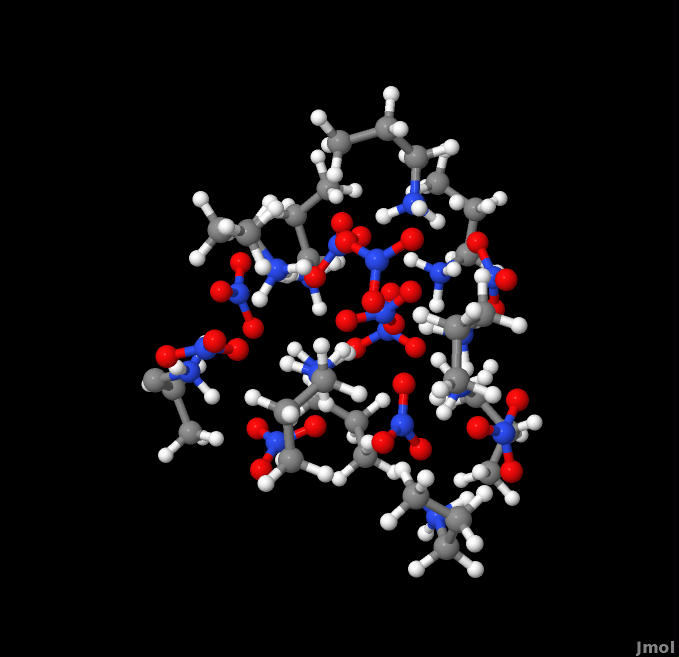}
 \caption{The actual conformation (left) and Douglas--Rachford reconstruction (right) of PAN. Note the two poorly reconstructed hydrogen atoms (white) in the left configuration.}\label{fig:ILrecon}
\end{figure}

\section{Concluding Remarks}\label{s:conclusion}
We have shown that the Douglas--Rachford reflection method can successfully solve the protein conformation determination problem by directly addressing a non-convex matrix completion problem. This is also the case for an analogous ionic liquid bulk structure determination problem. It is worth emphasising again that the current literature provides no theoretical justification for the method to work at all, let alone so well. Modifications of the method have also been shown to reduce computational times without significantly effecting the quality of the results. This promising demonstration of the method begs further attention, both in improving theoretical understanding, and in the refinement and investigation of these and further applications.

\paragraph{{\bf Acknowledgements.}~The authors wish to thank Dr Alister Page for introducing us to the bulk structure determination problem, and for kindly sharing the PAN data set. The work of JMB is supported, in part, by the Australian Research Council. The work of MKT is supported, in part, by an Australian Postgraduate Award.}


\end{document}